\documentclass[11pt]{amsart}
\usepackage{amsmath,amsfonts,amsthm,amssymb,amscd, verbatim, graphicx}

\newcommand{\comments}[1]{}

\renewcommand{\leq}{\leqslant}
\renewcommand{\geq}{\geqslant}

\newcommand{\GammaL}{\mathrm{\Gamma L}}
\newcommand{\GL}{\mathrm{GL}}
\newcommand{\SL}{\mathrm{SL}}
\newcommand{\bF}{\mathbb{F}}
\newcommand{\bK}{\mathbb{K}}
\newcommand{\bL}{\mathbb{L}}

\newtheorem{prop}{Proposition}
\newtheorem{thm}[prop]{Theorem}
\newtheorem{conj}{Conjecture}
\newtheorem{cor}[prop]{Corollary}
\newtheorem{lem}[prop]{Lemma}

\theoremstyle{definition}

\begin{document}

\title{On a conjecture of Degos}

\author{Nick Gill}
\address{
Department of Mathematics\\
University of South Wales\\
Treforest, CF37 1DL\\
U.K.}
\email{nicholas.gill@southwales.ac.uk}

\begin{abstract}
In this note we prove a conjecture of Degos concerning groups generated by companion matrices in $\GL_n(q)$.
\end{abstract}

\maketitle

Let $\bF$ be a field, and let $f\in\bF[X]$ be a polynomial of degree $n$, i.e.
\[f(X)=a_nX^n+a_{n-1}X_{n-1}+\cdots+a_1X+a_0\] where $a_0,\dots, a_n\in\bF$. Recall that the {\it companion matrix} of $f$ is the $n\times n$ matrix
\[
 C_f:=\begin{bmatrix}
  0 & \cdots & \cdots & \cdots & 0 & -a_0 \\
  1 & 0 & & & 0 & -a_1 \\
  0 & 1 & 0 &  & 0 & -a_2 \\
  \vdots & \ddots & \ddots & \ddots & \vdots & \vdots \\
  \vdots && \ddots & 1 & 0 & -a_{n-2} \\
   0&\cdots&\cdots& 0 & 1 & -a_{n-1} 
 \end{bmatrix}.
\]
The matrix $C_f$ has the property that its minimal polynomial and its characteristic polynomial are both equal to $f$. Conversely, if $g\in \GL_n(\bF)$ has minimal polynomial and characteristic polynomial both equal to some polynomial $f$, then $g$ is conjugate in $\GL_n(\bF)$ to $C_f$.

Recall in addition that if $\bF$ has order $q$ and $f\in\bF[X]$ has degree $n$, then $f$ is called {\it primitive} if it is the minimal polynomial of a primitive element $x\in \bF$. In \cite{degos}, J.-Y.~Degos makes the following conjecture.

\begin{conj}\label{c: degos}
Let $\bF$ be a field of order $p$ a prime, let $g=X^n-1$ and let $f\in\bF[X]$ be a primitive polynomial of degree $n$. Then $\langle C_f, C_g\rangle=\GL_n(p)$.
\end{conj}

We will prove a stronger version of this conjecture. Specifically, we prove the following.

\begin{thm}\label{t: main}
Let $\bF$ be a finite field of order $q$ and let $f, g\in\bF[X]$ be distinct polynomials of degree $n$ such that $f$ is primitive, and the constant term of $g$ is non-zero. Then $\langle C_f, C_g\rangle=\GL_n(q)$.
\end{thm}

For the rest of this paper $\bF$ is a finite field of order $q$.

\section{Field-extension subgroups}\label{s: fe}

Let $\bK=\bF(\alpha)$ be an algebraic extension of $\bF$ of degree $d$. Let $W=\bK^a$, and observe that $W$ is both an $a$-dimensional vector space over $\bK$ and an $ad$-dimensional space over $\bF$. 

A {\it $\bK/\bF$-semilinear automorphism} of $W$, $\phi$, is an invertible map $\phi:W\to W$ for which there exists $\sigma \in {\rm Gal}(\bK/\bF)$ such that, for all $v_1,v_2\in W$ and $k_1,k_2\in \bK$,
\[
\phi(k_1v_1+k_2v_2)=k_1^\sigma \phi(v_1)+ k_2^\sigma \phi(v_2).
\]
We define a group
\[
\GammaL_{\bK/\bF}(W)=\{ \phi: W \to W \mid \phi\textrm{ is a $\bK/\bF$-semilinear automorphism of }W\}.
\]
The group $\GammaL_{\bK/\bF}(W)$ can be written as a product $\GL_{a}(\bK).F$ where $F$ is a cyclic group of degree $d$ generated by the automorphism
\[
W \to W, \, \, (w_1, \dots, w_d)\mapsto (w_1^q, \dots, w_d^q).
\]
We will refer to elements of $F$ as {\it field-automorphisms} of $W$.

Now, for $\mathcal{B}=\{v_1,\dots, v_{ad}\}$ an ordered $\bF$-basis of $W$ and $\phi\in \GammaL_{\bK/\bF}(W)$, we define the following matrix
\[
 (\phi)_\mathcal{B}=\left[\begin{array}{c|c|c|c}
                     \phi(v_1) &  \phi(v_2) & \cdots & \phi(v_{ad})
                    \end{array}\right].
\]
It is a well-known fact that the map
\[
\Phi_\mathcal{B}: \GammaL_{\bK/\bF}(W) \to \GL_{ad}(q), \phi\mapsto (\phi)_\mathcal{B}
\]
is a well-defined injective group homomorphism, the image of which is a group $E$ known as a {\it field-extension subgroup of degree $d$} in $\GL_{ad}(q)$. Indeed, more is true: if we define
\[
 \theta: W \to \bF^{ad}, w\mapsto [w]_\mathcal{B},
\]
and consider $\Phi_\mathcal{B}$ to be a map $\GammaL_{\bK/\bF}(W) \to E$, then the pair $(\Phi,\theta)$ is a permutation group isomorphism. (Here, and throughout this note, we consider groups acting on the left.)

Note that the group $\GammaL_{\bK/\bF}(W)$ contains a unique normal subgroup $N$ isomorphic to $\GL_a(\bK)$. Then $H=\Phi_\mathcal{B}(N)$ is a subgroup of $\GL_{ad}(q)$ isomorphic to $\GL_a(\bK)$ and, writing $G=\GL_{ad}(q)$, one can check that $N_G(H)=E$, the associated field-extension subgroup. (To see this, note, firstly, that $E\leq N_G(H)\leq N_G(Z(H))$; now \cite[Proposition 4.3.3 (ii)]{kl} asserts that $N_G(Z(H)))=E$ and we are done.)

\section{Singer cycles}

Recall that a {\it Singer subgroup} of the group $\GL_n(q)$ is a cyclic subgroup of order $q^n-1$. In this section we prove the following lemma.

\begin{lem}\label{l: singer}
Let $g\in \GL_n(q)$ and let $f$ be its minimal polynomial. Then $\langle g\rangle$ is a Singer subgroup if and only if $f$ is primitive of degree $n$.

What is more, if $S=\langle g\rangle$ is a Singer subgroup, then $\langle g\rangle$ is conjugate to $\langle C_f\rangle$, and $S=\Phi_{\mathcal{B}}(GL_1(\mathbb{K}))$, where $\bK$ is a degree $n$ extension of $\bF$, and $\mathcal{B}$ is an ordered $\bF$-basis of $\bK$.
\end{lem}

\begin{proof}
Suppose that $S=\langle g\rangle$ is a Singer subgroup. Then $g$ contains an eigenvalue $\alpha$ that lies in $\bK$, a degree $n$ extension of $\bF$, and no smaller field. What is more, since $g$ has order $q^n-1$, so does $\alpha$ and so the minimal polynomial of $g$ is primitive of degree $n$ as required.

Suppose, on the other hand, that $f$ is primitive of degree $n$. 
Then the eigenvalues of $g$ are $\alpha, \alpha^q, \dots, \alpha^{q^{n-1}}$; in particular they are all distinct. Elementary linear algebra implies that $g$ is conjugate to $C_f$, the companion matrix of $f$. It is enough, then, to prove that $\langle C_f\rangle$ is a Singer cycle.

Let $\alpha$ be a primitive element of degree $n$ over $\bF$ and a root of $f$; let $\bK=\bF(\alpha)$, an extension of $\bF$ of degree $n$. We construct a field-extension subgroup $G$ of degree $n$ in $\GL_n(q)$ as the image of the map $\Phi_\mathcal{B}:\GammaL_{\bK/\bF}(\bK) \to \GL_n(q)$ where $\mathcal{B}=\{\alpha, \alpha^2,\dots, \alpha^{n-1}\}$.

By construction $H$ is isomorphic to $\GammaL_{\bK/\bF}(\bK)$ and, in particular, contains a subgroup isomorphic to $\GL_1(\bK)\cong \bK^*$. This subgroup is cyclic of order $q^n-1$ and is generated by the invertible linear transformation
\[
L_\alpha : \bK \to \bK, x \mapsto \alpha\cdot x.
\]
Now our construction guarantees that $\Phi_\mathcal{B}(L_\alpha)=C_f$ and we conclude, as required, that $C_f$ generates a cyclic subgroup of $\GL_n(q)$ of order $q^n-1$. In fact we have shown that $\langle C_f\rangle=\Phi_{\mathcal{B}}(GL_1(\mathbb{K}))$ and the final statement follows.
\end{proof}



\section{Two companion matrices}

\begin{lem}
 Let $H$ be a field-extension subgroup of degree $a$ in $\GL_{ad}(q)$. A non-trivial element of $H$ fixes at most $(q^a)^{d-1}$ elements of $V=(\bF)^{ad}$.
\end{lem}
\begin{proof}
We observed in \S\ref{s: fe} that the action of $H$ on $V$ is isomorphic to the action of $\GammaL_{\bK/\bF}(W)$ on $W=\bK^a$ where $\bK$ is a degree $d$ extension of $\bF$. Thus we set $\phi$ to be a non-trivial element of $\GammaL_{\bK/\bF}(W)$.
 
 If $\phi$ lies in $\GL_a(\bK)$ and is non-trivial, then basic linear algebra implies that the fixed-point set is a proper $\bK$-subspace of $W$ and so fixes at most $(q^a)^{d-1}$ elements of $W$.

Suppose that $\phi$ does not lie in $\GL_a(\bK)$. Thus we can write $\phi=h\sigma$ where $h$ is linear and $\sigma$ is a non-trivial field automorphism of $W$ that fixes $(\bF)^a$. 

Thus if $v\in \bK^a$ and $v^\phi=v$ we obtain immediately that $v^h=v^{\sigma^{-1}}$. Now if $c$ is a scalar that is not fixed by $\sigma$, then we obtain immediately that $(cv)^h\neq (cv)^{\sigma^{-1}}$. Since $v$ and $c$ were arbitary we conclude immediately that $g$ fixes at most $(q^b)^d$ elements where $b$ is some proper-divisor of $a$. The result follows.
 \end{proof}

\begin{cor}\label{c: two}
 If $C_f$ and $C_g$ are companion matrices of distinct monic polynomials $f,g \in\bF[x]$ of degree $n$, then $\langle C_f,C_g\rangle$ does not lie in a field-extension subgroup of $\GL_n(q)$.
\end{cor}
\begin{proof}
We consider the action of $\GL_n(q)$ on $V=\bF^n$. 
Observe that the images of the first $n-1$ elementary basis vectors are the same for both $C_f$ and $C_g$. In particular, then, the matrix $C_f^{-1}C_g$ fixes the $\bF$-span of these $n-1$ vectors and so fixes at least $q^{n-1}$ vectors. The previous lemma implies that, since $C_f\neq C_g$, we can conclude that $\langle C_f,C_g\rangle$ is not a subgroup of a field-extension subgroup of $\GL_{n}(q)$. 
\end{proof}

\section{A result about subgroups}

To complete the proof of Theorem~\ref{t: main} we will need the result below, Theorem~\ref{t: kantor}. In an earlier draft of this article, we attributed this result to Kantor \cite{kantor}. We are grateful to Peter Mueller who pointed out that Kantor's result relies on another paper -- \cite{cameron_kantor} -- which has subsequently been found to contain a number of errors.

In fact it is clear that the errors in \cite{cameron_kantor} are not fatal and that, with a little adjustment, the result still holds \cite{cameron_kantor_blog}. However, since no proof exists in the literature, we will sketch one below. Our approach uses a theorem of Hering \cite{hering}, a proof of which can be found in \cite[Appendix 1]{liebeck}. The disadvantage of our proof is that it relies on the Classification of Finite Simple Groups (CFSG), which Kantor's original approach did not.

\begin{lem}\label{l: unique}
 Suppose that $S$ is a Singer cycle in $\GL_n(q)$. Then, for each integer $d$ dividing $n$, there is a unique field-extension subgroup $\Phi_\mathcal{B}(\GammaL_{\bK/\bF}(W))$ (where $\bK$ is a field extension of $\bF$ of degree $d$) that contains $S$.
\end{lem}
\begin{proof}
Let $H$ be a subgroup of $\GL_n(q)$ that contains $S$ and suppose that $H\cong \GL_{n/d}(q^{d})$ for some divisor $d$ of $n$. Now $S$ is a Singer cycle in $H$ and so $S= \Phi_{\mathcal{C}}(\GL_1(\bL))$ where $\bL$ is a degree $n/d$ extension of $\bF_{q^d}$.

Write $Z$ for the unique subgroup of $S$ of order $q^d-1$. Direct calculation confirms that $Z$ coincides with the center of $H$. Thus $H\leq C_{\GL_n(q)}(Z)$. But $Z$ is precisely the $\bF_{q^d}$-scalar maps on $\bL$, and so (as we saw earlier, using \cite[Proposition 4.3.3(ii)]{kl}) $N_{\GL_n(q)}(Z)$ is a field-extension subgroup $\Phi_\mathcal{B}(\GammaL_{\bK/\bF}(\bL))$ where $\bK$ is a field extension of $\bF$ of degree $d$. But now $H$ must be the unique normal subgroup of this field-extension subgroup that is isomorphic to $\GL_{n/d}(q^d)$ and we are done.
\end{proof}

In the proof above we refer to two ordered $\bF$-bases of $\bL$, namely $\mathcal{B}$ and $\mathcal{C}$. It is an easy exercise to see that we can take $\mathcal{B}$ to be equal to $\mathcal{C}$.

\begin{thm}\label{t: kantor}
Let $L$ be a proper subgroup of $G=\GL_n(q)$  that contains a Singer cycle. Then $L$ contains a normal subgroup $H$ isomorphic to $\GL_{a}(q^c)$ with $n=ac$ and $c>1$. What is more $H$ is equal to $\Phi_\mathcal{B}(\GL_{a}(\bK))$ for $\bK$ some field extension of $\bF$ of degree $c$, and $\mathcal{B}$ some ordered $\bF$-basis of $\bK^{a}$.
\end{thm}


\begin{proof}
It is convenient, first, to deal with the case when $n=2$. If $L$ lies inside the normalizer of a non-split torus, then $L$ contains a normal subgroup $H\cong \GL_1(q^2)$, as required. Furthermore, order considerations imply that $L$ is a subgroup of neither the normalizer of a split torus, nor a Borel subgroup of $\GL_2(q)$.

The remaining subgroups of $\GL_2(q)$ can be deduced from a classical theorem of \cite{dickson}. In particular, $L\cap \SL_2(q)$ is isomorphic to either $A_4, S_4, A_5$ or a double cover of one of these. In particular the maximal order of an element of $L\cap \SL_2(q)$ is $10$. Since $L\cap \SL_2(q)$ must contain an element of order $q+1$, we conclude that $q\leq 9$. Now computation in the remaining groups (using, for example, \cite{gap}) rules out the remaining possibilities.

Assume, then that $n\geq 3$, and we refer to Hering's Theorem, as presented in \cite[Appendix 1]{liebeck}. This result lists those subgroups of $\GL_{\ell}(p)$  (for $\ell\in\mathbb{Z}^+$) that act transitively on the set of non-zero vectors of $(\mathbb{F}_p)^{\ell}$. Since $G$ embeds naturally (inside a field extension subgroup) in $\GL_{\ell}(p)$ for $\ell=n\log_pq$ and, since a Singer cycle acts transitively (via this embedding) on the set of non-zero vectors in $(\mathbb{F}_p)^{\ell}$, this list contains all the possible groups $L$. In what follows we fix a field-extension embedding 
\[
\Phi_{\mathcal{D}}: G \hookrightarrow \GL_{\ell}(p)
\]
for $\ell=n\log_pq$, and $\mathcal{D}$ an ordered $\bF_p$-basis of $(\bF)^n$. We obtain an associated action on the vector space $V=(\mathbb{F}_p)^{\ell}$, and apply the theorem. 

According to Hering's Theorem, the group $L$ lies in one of three class (A), (B) and (C). Given that $\ell\geq n\geq 3$, the classes (B) and (C) reduce to the following possibilities:
\begin{enumerate}
\item $L=A_6, A_7$ or $\SL_2(13)$; $G=\GL_4(2), \GL_6(3)$ or $\GL_3(9)$.
\item $L$ has a normal subgroup $R\cong D_8\circ Q_8$, $L/R\leq S_5$ and $G=\GL_4(3)$.
\end{enumerate}
In the first case, we note that all elements of $L$ have order less than or equal to $14$, and this case is immediately excluded. Similarly, in the second case, all elements of $L$ have order less than or equal to $48$, and this case is immediately excluded.

We are left with groups in Liebeck's class A. These come in four families; we examine them one at a time. For family (1), $L$ is a subgroup of the normalizer of a Singer cycle. The result follows immediately in this case. For the remaining families, $L$ has a normal subgroup $N$ isomorphic to $\SL_a(q_0)$, ${\rm Sp}_{a}(q_0)$ or $G_2(q_0)$ with $q_0=p^d$ and $\ell=ad$.

By examining the proof in \cite{liebeck}, we find that, in all cases, $L$ lies in a field-extension subgroup $\Phi_\mathcal{C}(\GammaL_{\bK_0/\bF_p}(W))$ of $\GL_\ell(p)$, for $\bK_0$ some field extension of $\bF_p$ of degree $d\in \mathbb{Z}^+$ and $\mathcal{C}$ some ordered $\bF_p$-basis of $W=(\bK_0)^{a}$. What is more $q_0=q^{d}$ and $N\leq \Phi_\mathcal{C}(\GL_{a}(\bK_0))$.

In the symplectic case, this means that the action of $N$ on $(\bK_0)^a$ yields the natural module for ${\rm Sp}_a(\bK_0)$ (see, for instance, \cite[Proposition 5.4.13]{kl}. Now one can check that an irreducible cyclic subgroup of ${\rm Sp}_a(q_0)$ in the natural module has size dividing $q_0^{a/2}+1$ (see, for instance, \cite{bereczky}). Now Schur's Lemma implies that an irreducible cyclic subgroup of $L$ has order dividing ${(q_0^{a/2}+1)2(q_0-1)\log_p(q_0)}$. Since this must be at least $q_0^a-1$, one immediately obtains that $a/2=1$ and, since ${\rm Sp}_2(\bK_0)\cong \SL_2(\bK_0)$ we are in one of the remaining cases.

If $G=G_2(q_0)$, then the proof in \cite{liebeck} implies that, in fact, $N$ is a subgroup of a symplectic group ${\rm Sp}_6(q_0)$ that acts on $(\bK_0)^6$ via its natural module. Thus this situation can be excluded via the calculation of the previous paragraph.

We are left with the case where 
\[
N\cong \SL_a(q_0)\lhd L \leq \Phi_\mathcal{C}(\GammaL_{\bK_0/\bF_p}(W))\leq \GL_\ell(p).
\]
Direct computation inside $\GammaL_{\bK_0/\bF_p}(W)$ confirms that, since $L$ contains a cyclic group of order $p^\ell-1$, $L$ must contain $M=\Phi_\mathcal{C}(\GL(W))\cong GL_a(q_0)$ as a normal subgroup.

Observe, then, that the Singer cycle $S$ lies in two field extension subgroups of $\GL_d(p)$, namely $N_{\GL_d(p)}(G)$ and $N_{\GL_d(p)}(M)$.  Notice, though, that by Lemma~\ref{l: singer}, $S=\Phi_\mathcal{B}(GL_1(\bL))$ for some ordered $\bF_p$-basis $\mathcal{B}$ of $\bL$, a degree $n$ extension of $\bF_p$. Clearly the groups $\Phi_\mathcal{B}(\GammaL_{\bF/\bF_p}(\bL))$ and $\Phi_\mathcal{B}(\GammaL_{\bK_0/\bF_p}(\bL))$ are also field extension subgroups that contain $S$.

Now Lemma~\ref{l: unique} implies that $M=\Phi_\mathcal{B}(\GL_a(\bK_0))$ and $G=\Phi_\mathcal{B}(\GL_n(\bF))$. The second occurrence of the monomorphism $\Phi_\mathcal{B}$ here is simply a restriction of the first; it is an easy exercise to check that, in this situation, $M$ is a field-extension subgroup of $G$ as required.
\end{proof}

\section{Proving Theorem~\ref{t: main}}

Observe that if $f$ and $g$ are as in Theorem~\ref{t: main}, then they both have non-zero constant term and hence are invertible and so lie in $\GL_n(q)$. Now Lemma~\ref{l: singer}, Corollary~\ref{c: two} and Theorem~\ref{t: kantor} imply that $\langle C_f, C_g\rangle$ does not lie in a proper subgroup of $\GL_n(q)$. In other words $\langle C_f, C_g\rangle = \GL_n(q)$, as required.


\begin{thebibliography}{{Her}85}

\bibitem[Ber00]{bereczky}
{\'A}.~Bereczky.
\newblock Maximal overgroups of {S}inger elements in classical groups.
\newblock {\em J. Algebra}, 234(1):187--206, 2000.

\bibitem[Cam]{cameron_kantor_blog}
P.~J. Cameron.
\newblock Antiflag-transitive groups.
\newblock 2015. Blogpost at:\\ {\tt
  https://cameroncounts.wordpress.com/2015/05/31/antiflag-transitive-groups/}.

\bibitem[CK79]{cameron_kantor}
P.~J. {Cameron} and W.~M. {Kantor}.
\newblock {2-transitive and antiflag transitive collineation groups of finite
  projective spaces.}
\newblock {\em {J. Algebra}}, 60:384--422, 1979.

\bibitem[Deg13]{degos}
J.-Y. Degos.
\newblock Linear groups and primitive polynomials over {${\bf F}_p$}.
\newblock {\em Cah. Topol. G\'eom. Diff\'er. Cat\'eg.}, 54(1):56--74, 2013.

\bibitem[{Dic}58]{dickson}
L.~E. {Dickson}.
\newblock {Linear groups. With an exposition of the Galois field theory}, 1958.

\bibitem[GAP15]{gap}
The GAP~Group.
\newblock {\em {GAP -- Groups, Algorithms, and Programming, Version 4.7.8}},
  2015.

\bibitem[{Her}85]{hering}
C.~{Hering}.
\newblock {Transitive linear groups and linear groups which contain irreducible
  subgroups of prime order. II.}
\newblock {\em {J. Algebra}}, 93:151--164, 1985.

\bibitem[Kan80]{kantor}
W.~M. Kantor.
\newblock Linear groups containing a {S}inger cycle.
\newblock {\em J. Algebra}, 62:232--234, 1980.

\bibitem[KL90]{kl}
P.~Kleidman and M.~Liebeck.
\newblock {\em The subgroup structure of the finite classical groups}, volume
  129 of {\em London Mathematical Society Lecture Note Series}.
\newblock Cambridge University Press, Cambridge, 1990.

\bibitem[{Lie}87]{liebeck}
M.~W. {Liebeck}.
\newblock {The affine permutation groups of rank three.}
\newblock {\em {Proc. Lond. Math. Soc. (3)}}, 54:477--516, 1987.

\end{thebibliography}
\end{document}